\documentclass[a4paper,12pt]{article}

\usepackage{ifthen}
\usepackage{amssymb}
\usepackage{amsmath}
\usepackage{MnSymbol}
\usepackage{stmaryrd}
\usepackage{mathrsfs}
\usepackage{theorem}
\usepackage[all]{xypic}
\usepackage{color}
\usepackage{graphicx}
\usepackage{xparse}

\usepackage[affil-it]{authblk}
\usepackage[backend = biber,  maxbibnames=9, maxcitenames=9]{biblatex}
\makeatletter
\def\@maketitle{%
  \newpage
  \null
  \vskip 2em%
  \begin{center}%
  \let \footnote \thanks
    {\Large\bfseries \@title \par}%
    \vskip 1.5em%
    {\normalsize
      \lineskip .5em%
      \begin{tabular}[t]{c}%
        \@author
      \end{tabular}\par}%
    \vskip 1em%
    {\normalsize \@date}%
  \end{center}%
  \par
  \vskip 1.5em}
\makeatother

	{
		\theorembodyfont{\itshape}

		\newtheorem{theorem}{Theorem}[section]
		\newtheorem{lemma}[theorem]{Lemma}
		\newtheorem{proposition}[theorem]{Proposition}
		\newtheorem{corollary}[theorem]{Corollary}
	}

	{
		\theorembodyfont{\rmfamily}
		\newtheorem{definition}[theorem]{Definition}
		
		\newtheorem{remark}[theorem]{Remark}
		\newtheorem{question}[theorem]{Question}
		\newtheorem{example}[theorem]{Example}
	}

	\newenvironment{proof}{
		\goodbreak\par
		\textit{Proof.}%
	}{%
		\nopagebreak
		\hfill{\vrule width 1ex height 1ex depth 0ex}
		\medskip
		\goodbreak
	}
	\newcommand{\sizedescriptor}[2]
	{
		\ifthenelse{\equal{#1}{0}}{}{
		\ifthenelse{\equal{#1}{1}}{\big}{
		\ifthenelse{\equal{#1}{2}}{\Big}{
		\ifthenelse{\equal{#1}{3}}{\bigg}{
		\ifthenelse{\equal{#1}{4}}{\Bigg}{
		#2}}}}}
	}

	\newcommand{\proven}[1]{\underline{#1}\vspace{0.2em}\\}

	\newcommand{\ie}[1][~]{i.e.{#1}}

	\newcommand{\df}[1]{{\bf{#1}}}
	
	\newcommand{\equ}{\sim}
	\newcommand{\dfeq}{\mathrel{\mathop:}=}

	\newcommand{\sepdfeq}{\quad\dfeq\quad}
	
	\newcommand{\impl}{\Rightarrow}
	\newcommand{\revimpl}{\Leftarrow}
	\newcommand{\lequ}{\Leftrightarrow}

	\newcommand{\all}[4][auto]{\forall\, #2 \,{\in}\, #3\,.\sizedescriptor{#1}{\left}({#4}\sizedescriptor{#1}{\right})}

	\newcommand{\xsome}[3]{\exists\, #1 \,{\in}\, #2\,.\,#3}

	\newcommand{\st}[3][auto]{\sizedescriptor{#1}{\left}\{#2\;\sizedescriptor{#1}{\middle}|\;#3\sizedescriptor{#1}{\right}\}}
	\newcommand{\NN}{\mathbb{N}}

	\newcommand{\RR}{\mathbb{R}}

	\newdir{ >}{{}*!/-8pt/@{>}}% used for monomorphisms:  \ar@{ >->} is a mono
	\newdir{ (}{{}*!/-5pt/@^{(}}% used for inclusions: \ar@{ (->} is an inclusion
	\newdir{|>}{!/5pt/@{|}*:(1,-.2)@{>}*:(1,+.2)@_{>}}% used for covers: \ar@{-|>} is a cover

	\newcommand{\el}[2][n]{\mathtt{el}_{#1,#2}}%{\overset{#1}{\mathtt{el}}_{#2}}

	\newcommand{\mms}{\varsigma}

\title{Symmetric Polynomials in Upper-bound Semirings}

\date{}
\author{Sara Kali\v{s}nik}
\affil{Max Planck Institute for Mathematics in the Sciences, sara.kalisnik@mis.mpg.de}
\affil{Wesleyan University, skalisnikver@wesleyan.edu}

\author{Davorin Le\v{s}nik%
  \thanks{Electronic address: \texttt{davorin.lesnik@fmf.uni-lj.si}; This author was partially supported by the Air Force Office of Scientific Research, Air Force Materiel Command, USAF under Award No.~FA9550-14-1-0096.}}
\affil{Department of Mathematics, University of Ljubljana}

\begin{document}
	
	\maketitle
	\vspace{-1cm}
	\begin{abstract}
		The fundamental theorem of symmetric polynomials over rings is a classical result which states that every unital commutative ring is fully elementary, \ie we can express symmetric polynomials with elementary ones in a unique way. The result does not extend directly to polynomials over semirings, but we do have analogous results for some special semirings, for example, the tropical, extended and supertropical semirings. These all fall into a larger class of upper-bound semirings. In this paper we extend the known results and give a complete characterization of fully elementary upper-bound semirings. We further improve this characterization statement in the case of linearly ordered upper-bound semirings.
	\end{abstract}

	\section{Introduction}\label{SECTION: Introduction}
	
		The fundamental theorem of symmetric polynomials states that any symmetric polynomial in variables $x_1,\ldots , x_n$ over a unital commutative ring can be represented in a unique way as a polynomial in the elementary symmetric polynomials,
		\begin{align*}
			\el{1}(x_1, \ldots, x_n)  &= x_1 +\ldots + x_n\\
			\el{2}(x_1, \ldots, x_n) &= x_1 x_2 + x_1 x_3 + x_2 x_3 + \ldots + x_{n-1} x_n,\\
			\vdots\\
			\el{n}(x_1, \ldots, x_n) &= x_1 x_2 \ldots x_n.
		\end{align*}	
		
		This theorem has an interesting history~\cite{BSC}: Newton already extensively studied symmetric root polynomials as early as the mid-1660s, and was very likely aware of the fundamental theorem at that time. Over the course of the eighteenth century many mathematicians seemed familiar with it and used it. For example, Lagrange claimed it was `self-evident'. The existence part of the proof only appeared in 1771, when Vandermonde stated and proved the result in terms of roots and coefficients of a polynomial. The uniqueness aspect was still neglected at the time. The first complete proof is due to Dedekind.
		
		Given the long tradition and the impact of this theorem it is surprising that the question of elementarity over semirings (\ie the ability to express symmetric polynomials with elementary ones) only recently started being researched. The first paper on the topic analyzed the tropical semiring~\cite{symtrop}. The motivation for doing that came from the study of persistent homology~\cite{elz-tps-02, topodata}, a method which assigns a barcode, \ie a collection of intervals, to a finite metric space and can be used as a measurement of the shape of data. Studying symmetric tropical functions turned out to be crucial in identifying well-behaved coordinates on the space of barcodes~\cite{KV2016, MKV2016}.
		
		A positive result in the case of the tropical semiring led to a discussion about what happens in other semirings of interest to tropical algebraists, such as the symmetrized max-plus semiring~\cite{gaubertmaxplus}, the extended tropical semiring~\cite{extendedsemiring}, the supertropical semiring~\cite{Izhakian20102222, extendedsemiring, Izhakian2011}. In~\cite{KVL2017} we proved that supertropical semirings, including the extended tropical semiring, are elementary. The symmetrizations of the tropical and max-plus semiring, however, are not. 
		
		All of these semirings are upper-bound, so in~\cite{KVL2017} we already presented some partial results about elementarity for upper-bound semirings, though we have not yet obtained a proper characterization result. This is one of the main contributions of this paper.  
		
		This manuscript is organized as follows. In Section~\ref{SECTION: Preliminaries} we review the basic definitions needed to state and prove theorems. In Section~\ref{SECTION: Elementarity in Upper-bound Semirings} we study elementarity in upper-bound semirings. The central result in this section is Theorem~\ref{THEOREM: Characterization of elementary upper-bound semirings}, which gives a complete characterization of fully elementary upper-bound semirings as the Frobenius ones. In Section~\ref{SECTION: Elementarity in Linearly Ordered Semirings} we consider linearly ordered upper-bound semirings and improve the characterization in this case. In Section~\ref{SECTION: Towards General Elementarity} we consider what the scope of our results is from the perspective of general unital commutative semirings, rather than just the upper-bound ones.

	\section{Preliminaries}\label{SECTION: Preliminaries}
	
		Recall that $(X,+,0,\cdot)$ is a \df{semiring} when $(X,+,0)$ is a commutative monoid, $(X,\cdot)$ a semigroup, the multiplication $\cdot$ distributes over the addition $+$ and $0$ is an absorbing element, \ie $0 \cdot x = x \cdot 0 = 0$ for all $x \in X$. A semiring is \df{unital} when it has the multiplicative unit $1$. A semiring is \df{commutative} when $\cdot$ is commutative. A semiring is \df{idempotent} when $+$ is idempotent, \ie $x + x = x$ for all $x \in X$. A map between semirings is a semiring homomorphism when it preserves addition, additive unit and multiplication. If its domain and codomain are unital semirings, it is called a unital semiring homomorphism when it additionally preserves the multiplicative unit.
		
		We often omit the $\cdot$ sign in algebraic expressions, and shorten the product of $n$ many factors $x$ to $x^n$. Also, we write $X$ instead of $(X,+,0,\cdot)$ (or $(X,+,0,\cdot,1)$) when the operations are clear.
		
		Given a unital semiring $X$, we can view any natural number\footnote{We treat $0$ as a natural number, so $\NN = \{0, 1, 2, 3, \ldots\}$.} $n \in \NN$ as an element of $X$ in the usual way:
		\[n = \underbrace{1 + 1 + \ldots + 1}_{\text{$n$ times}}.\]
		However, this mapping (in fact, a unital semiring homomorphism) from $\NN$ to $X$ need not be injective; for example, if (and only if) $X$ is idempotent, we have $1 = 2$ in $X$.
		
		In this paper we almost exclusively deal with \df{unital commutative semirings} which we shorten to \df{uc-semirings}.
		
		Any commutative monoid $X$, and thus in particular any semiring, has an \df{intrinsic order} $\leq$ (see~\cite{KVL2017} ), given by
		\[a \leq b \sepdfeq \xsome{x}{X}{a + x = b}.\]
		The intrinsic order is a preorder (reflexive and transitive) with $0$ a least element. The operations $+$ and $\cdot$ are monotone: if $a \leq b$ and $c \leq d$, then $a + c \leq b + d$ and $a \cdot c \leq b \cdot d$. Also, any semiring homomorphism is monotone with respect to the intrinsic order.
		
		The following are the two most crucial properties of $\leq$ that we use in this paper. First, the intrinsic order is \df{directed}, in the sense that any two elements $a, b \in X$ have an upper bound, namely $a + b$. In other words, adding summands increases the sum. The second crucial property for us is antisymmetry of $\leq$. We do not have it in general, though, so we recall the definition~\cite{KVL2017}.
		
		\begin{definition}
			A semiring is \df{upper-bound} when its intrinsic order is antisymmetric (thus a partial order).
		\end{definition}
		
		A simple example of an upper-bound semiring is the set of natural numbers $\NN$, where the intrinsic order is the usual one. Note also that every idempotent semiring is upper-bound.
		
		Recall that any preorder $\leq$ defines an equivalence relation by
		\[a \approx b \ \dfeq \ a \leq b \land b \leq a.\]
		A preorder is (by definition) antisymmetric (hence, a partial order) when $\approx$ is equality. In general, a preorder on a set $X$ induces a partial order on the quotient set $X/_\approx$. If $X$ is a semiring and $\leq$ its intrinsic order, then $\approx$ is a congruence (that is, if $a \approx b$ and $c \approx d$, then $a + c \approx b + d$ and $a \cdot c \approx b \cdot d$), meaning that we can turn any semiring into an upper-bound semiring by taking the quotient $X/_\approx$.
		
		There is, of course, no shortage of non-upper-bound semirings. For example, if $X$ is a ring, then for every $a, b \in X$ we have $a \leq b$ (and so $a \approx b$). Thus the only upper-bound ring is the trivial one $\{0\}$.
		
		We now turn our attention to polynomials over semirings.
		
		\begin{definition}\label{DEFINITION: Polynomials in semirings}
			Let $X$ be a uc-semiring.
			\begin{itemize}
				\item
					Let $m, n, d_{1,1}, \ldots, d_{m,n}$ be natural numbers and $a_1, \ldots, a_m \in X$. A \df{monomial} is a syntactic object of the form $a_k \prod_{j = 1}^n x_j^{d_{k,j}}$. A \df{polynomial} is a sum of monomials $\sum_{k = 1}^m a_k \prod_{j = 1}^n x_j^{d_{k,j}}$. In this paper we only consider polynomials over commutative semirings, so we treat polynomials as equal if they differ only in the order of summands and/or factors.
				\item
					A \df{polynomial function} is the function $X^n \to X$ that a polynomial represents.
				\item
					A polynomial is \df{symmetric} when for each monomial $a_k \prod_{j = 1}^n x_j^{d_{k,j}}$ in it and each permutation $\sigma \in S_n$ the monomial $a_k \prod_{j = 1}^n x_{\sigma(j)}^{d_{k,j}}$ also appears in it, up to a change of the order of factors.\footnote{That is, we consider symmetry relative to commutativity of multiplication. For example, the polynomial $x y$ is symmetric: the transposition of variables gives $y x$ which we identify with $x y$. Of course, it would not make sense to consider $x y$ symmetric over a non-commutative semiring.} A polynomial function is \df{symmetric} when it can be represented by a symmetric polynomial.
			\end{itemize}
		\end{definition}
		
		\begin{remark}\label{REMARK: variants of symmetry}
			There are two reasonable definitions of when a polynomial function is symmetric: if it is represented by a symmetric polynomial (let us say that a function is `syntactically symmetric' in this case), or if its values are invariant under arbitrary permutations of variables, that is, $p(x_1, \ldots, x_n) = p(x_{\sigma(1)}, \ldots, x_{\sigma(1)})$ for all $(x_1, \ldots, x_n)$ and permutations $\sigma \in S_n$ (say that $p$ is `semantically symmetric'). Every syntactically symmetric polynomial function is semantically symmetric, but we do not know whether the converse holds (for more on the subject, see~\cite{KVL2017} ). For the purposes of our theorems, a `symmetric polynomial function' refers to syntactic symmetry (as stated in Definition~\ref{DEFINITION: Polynomials in semirings}), since our proofs require considering polynomials on the syntactic level. 
		\end{remark}
		
		Let us have a more formal description of a symmetric polynomial (function). For any $n \in \NN$ and $d = (d_1, \ldots, d_n) \in \NN^n$ consider the following subset of the symmetric group $S_n$:
		\[S_n^d \dfeq \st{\pi \in S_n}{\all[1]{i, j}{\{1, \ldots, n\}}{(d_i = d_j \land i \leq j) \implies \pi(i) \leq \pi(j)}}.\]
		Then define the \df{minimal symmetric segment}, corresponding to $d$, as
		\[\mms(d) \dfeq \sum_{\pi \in S_n^d} \prod_{j = 1}^n x_{\pi(j)}^{d_j}.\]
		The idea is that $\mms(d)$ represents the minimal polynomial which is symmetric and has monomials in which the exponents of variables are given by $d$. For example, we have $\mms(2,1) = x_1^2 x_2 + x_2^2 x_1$ and $\mms(1,1) = x_1 x_2$. However, for the sake of legibility we prefer to use variables $x, y, \ldots$ instead of $x_1, x_2, \ldots$ in concrete cases such as these, so we write for example $\mms(2,2,0,0) = x^2 y^2 + x^2 z^2 + x^2 w^2 + y^2 z^2 + y^2 w^2 + z^2 w^2$.
		
		Clearly $\mms(d)$ does not change if we permute the terms of $d$, so whenever necessary we can assume without loss of generality that $d$ is a decreasing sequence. This is helpful since if we view $\mms$ as a map from finite decreasing sequences to polynomials, it is injective.
		
		So how many terms does $\mms(d)$ have? From each group of terms which differ only by the order of factors we choose a single term (specifically, the one with the exponents in the order, given by $d$, and in which among the same exponents the variables have increasing indices). For example, $\mms(2,1) = x_1^2 x_2 + x_2^2 x_1$ contains one of the terms $x_1^2 x_2$, $x_2 x_1^2$ and one of the terms $x_2^2 x_1$, $x_1 x_2^2$. In essence, we are considering the permutations of multisets. That is, if we split $d$ into blocks of equal terms and the sizes of these blocks are $k_1, \ldots, k_j$, then $|S_n^d| = \frac{n!}{k_1! k_2! \ldots k_j!}$.
		
		The crucial observation here is the following: a polynomial (or a polynomial function) $p$ over a unital commutative semiring $X$ is symmetric if and only if it can be written as a linear combination of minimal symmetric segments, in the sense that
		\[p(x_1, \ldots, x_n) = \sum_{k = 1}^m a_k \, \mms(d_k)\]
		for some $m \in \NN$, $a_1, \ldots, a_m \in X$ and $d_1, \ldots, d_m \in \NN^n$.
		
		Recall that the fundamental theorem of symmetric polynomials over rings states that every symmetric polynomial function over a ring can be represented by a polynomial in elementary symmetric polynomials. Our main purpose in this paper is to prove a similar version of this theorem for upper-bound semirings.
		
		\begin{definition}
			Let $X$ be a uc-semiring.
			\begin{itemize}
				\item
					For any $n \in \NN$ and $k \in \NN_{\leq n}$ the corresponding \df{elementary symmetric polynomial} is defined as
					\[\el{k}(x_1, \ldots, x_n) \dfeq \mms(\underbrace{1, \ldots, 1}_{k \text{ many}}, \underbrace{0, \ldots, 0}_{n-k \text{ many}}).\]
					In other words, the elementary symmetric polynomials are precisely the minimal symmetric segments, given by (decreasing) finite binary sequences.\footnote{Note that $\el{0}$ is the constant $1$, and thus generally uninteresting in this context.}
				\item
					Given $n \in \NN$, we say that $X$ is \df{$n$-elementary} when for every symmetric polynomial $p$ in $n$ variables there exists a polynomial $r$ in $n$ variables, such that
					\[p(x_1, \ldots, x_n) = r\big(\el{1}(x_1, \ldots, x_n), \ldots, \el{n}(x_1, \ldots, x_n)\big)\]
					for all $x_1, \ldots, x_n \in X$ at the level of polynomial functions.
				\item
					$X$ is \df{fully elementary} when it is $n$-elementary for all $n \in \NN$.
			\end{itemize}
		\end{definition}
		
		Obviously, any uc-semiring is $0$-elementary and $1$-elementary. Note that higher elementarity implies lower one.
		
		\begin{proposition}\label{PROPOSITION: Higher elementarity implies lower one}
			Let a uc-semiring $X$ be $n$-elementary. Then it is also $m$-elementary for all $m \in \NN_{\leq n}$.
		\end{proposition}
		
		\begin{proof}
			It suffices to show for any $n \in \NN$ that if $X$ is $(n+1)$-elementary, it is also $n$-elementary. To this end, take any symmetric polynomial $p$ in $n$ variables in $X$; we can write it in the form
			\[p(x_1, \ldots, x_n) = \sum_{k = 1}^m a_k \mms(d_{k,1}, \ldots, d_{k,n}).\]
			Define
			\[q(x_1, \ldots, x_n, x_{n+1}) \dfeq \sum_{k = 1}^m a_k \mms(d_{k,1}, \ldots, d_{k,n}, 0).\]
			Then $q$ is symmetric as well, so by assumption there exists a polynomial $r$ with
			\[q(x_1, \ldots, x_{n+1}) = r\big(\el[n+1]{1}(x_1, \ldots, x_{n+1}), \ldots, \el[n+1]{n+1}(x_1, \ldots, x_{n+1})\big).\]
			We get
			\[p(x_1, \ldots, x_n) = q(x_1, \ldots, x_n, 0) =\]
			\[= r\big(\el[n+1]{1}(x_1, \ldots, x_n, 0), \ldots, \el[n+1]{n}(x_1, \ldots, x_n, 0), x_1 \ldots x_n \cdot 0\big) =\]
			\[= r\big(\el{1}(x_1, \ldots, x_n), \ldots, \el{n}(x_1, \ldots, x_n), 0\big) =\]
			\[= s\big(\el{1}(x_1, \ldots, x_n), \ldots, \el{n}(x_1, \ldots, x_n)\big)\]
			for $s(x_1, \ldots, x_n) \dfeq r(x_1, \ldots, x_n, 0)$.
		\end{proof}

	\section{Elementarity in Upper-bound Semirings}\label{SECTION: Elementarity in Upper-bound Semirings}
	
		In~\cite{KVL2017} we showed that in idempotent semirings full elementarity is equivalent to the fulfillment of Frobenius equalities, with the purpose to have a characterization of elementarity for the tropical semiring and semirings related to it. In this section we significantly generalize this result to upper-bound semirings.
		
		Recall that the \df{Frobenius equality} for the exponent $n \in \NN$ in a uc-semiring $X$ states that $(x + y)^n = x^n + y^n$ for all $x, y \in X$. The Frobenius equality for $0$ states $1 = 1 + 1$, so it is equivalent to the idempotency of the semiring, but we do not want to restrict ourselves to those. As such, we consider only the Frobenius equalities for positive exponents.
		
		\begin{definition}\label{DEFINITION: Frobenius}
			A uc-semiring $X$ is \df{Frobenius} when it satisfies Frobenius equalities for all $n \in \NN_{\geq 1}$.
		\end{definition}
		
		Putting natural numbers into Frobenius equalities reveals some properties of a Frobenius semiring. We prepare a lemma for later.
		
		\begin{lemma}\label{LEMMA: Natural number equalities in a semiring}
			Let $X$ be a unital semiring.
			\begin{enumerate}
				\item
					If $2 = 3$ in $X$, then $2 = 4$ in $X$.
				\item
					If $2 = 4$ in $X$ and $X$ is upper-bound, then $2 = 3$ in $X$.
				\item
					If $X$ is Frobenius, then $2 = 4$ in $X$.
			\end{enumerate}
		\end{lemma}
		
		\begin{proof}
			\begin{enumerate}
				\item
					Add $1$ to the equality $2 = 3$ and conclude $2 = 3 = 4$.
				\item
					We have $2 \leq 2 + 1 = 3 \leq 3 + 1 = 4 = 2$. Since $X$ is upper-bound, we infer $2 = 3$.
				\item
					Calculate $2 = 1^2 + 1^2 = (1 + 1)^2 = 4$.
			\end{enumerate}
		\end{proof}
		
		Note that if indeed $2 = 4$ in a unital semiring $X$, then any two $m, n \in \NN_{\geq 2}$ are equal whenever they have equal parity. If $2 = 3$, then all $m, n \in \NN_{\geq 2}$ are equal in $X$.
		
		We can now formulate and prove the first step towards the general theorem classifying elementarity in upper-bound uc-semirings. We calculate the product between an elementary symmetric polynomial $\el{k}$ and a minimal symmetric segment with at most $k$ non-zero exponents. It is instructive to start with an example.
		
		\begin{example}\label{EXAMPLE: Multiplying a minimal symmetric segment with an elementary symmetric polynomial}
			\[\mms(3,2,0,0) \cdot \mms(1,1,0,0) =\]
			\[= (x^3 y^2 + x^3 z^2 + x^3 w^2 + y^3 x^2 + y^3 z^2 + y^3 w^2 + z^3 x^2 + z^3 y^2 + z^3 w^2 + w^3 x^2 + w^3 y^2 + w^3 z^2) \cdot\]
			\[\cdot (x y + x z + x w + y z + y w + z w) =\]
			\[\text{\footnotesize (evaluate the 72 terms, collect them together to form minimal symmetric segments)}\]
			\[= \mms(4,3,0,0) + \mms(4,2,1,0) + 2 \cdot\mms(3,3,1,0) + \mms(3,2,1,1) =\]
			\[= \mms(3+1,2+1,0,0) + \mms(3+1,2,0+1,0) + 2 \cdot\mms(3,2+1,0+1,0) + \mms(3,2,0+1,0+1)\]
		\end{example}
		
		The point is, the product of $\mms(3,2,0,0)$ and $\mms(1,1,0,0)$ can be expressed as a linear combination (with natural numbers as coefficients) of minimal symmetric segments of sequences that we get by adding the terms of permutations of $(3,2,0,0)$ and $(1,1,0,0)$ and ordering the result to a decreasing sequence.
		
		Consider now the general case. Let $X$ be a uc-semiring, $n \in \NN$, $k \in \{1, \ldots, n\}$ and $d = (d_1, \ldots, d_n) \in \NN^n$ a decreasing sequence, for which all non-zero terms appear among the first $k$ terms (meaning, either $k = n$ or $k < n$ and $d_{k+1} = 0$).
		
		For each $j \in \{k, \ldots, n\}$ define $I^d_j$ to be the set of those binary sequences $\alpha \in \{0, 1\}^n$, such that:
		\begin{itemize}
			\item
				$\alpha$ has exactly $k$ many $1$s (therefore $n-k$ many $0$s),
			\item
				the index of the last $1$ in $\alpha$ is $j$,
			\item
				the (componentwise) sum $d + \alpha$ is a decreasing sequence.
		\end{itemize}
		Note that then for any $\alpha \in I^d_j$ and any $i \in \NN$ with $k < i \leq j$ we have $\alpha_i = 1$.
		
		Now, the polynomial $\mms(d) \cdot \el{k}$ is a product of two symmetric polynomials, and is therefore symmetric itself, meaning that we can write it as a linear combination of minimal symmetric segments. Consider an arbitrary term $x_{\pi(1)}^{d_1} \ldots x_{\pi(n)}^{d_n}$ from $\mms(d)$ and an arbitrary term $x_{\rho(1)} \ldots x_{\rho(k)}$ from $\el{k}$. Their product is the product of powers of variables, $n-k$ many of which have the exponent of the form $d_i$ and the remaining $k$ have the exponent of the form $d_i + 1$. We can always order the exponents decreasingly, so each such product of monomials appears in $\mms(d + \alpha)$ for some $j \in \{k, \ldots, n\}$ and $\alpha \in I^d_j$. A particular product of monomials can appear several times, but that number is necessarily a natural number since the coefficients of monomials in $\mms(d)$ and $\el{k}$ are all equal to $1$. That is, we have
		\[\mms(d) \cdot \el{k} = \sum_{j = k}^n \sum_{\alpha \in I^d_j} a_{j,\alpha} \, \mms(d + \alpha)\]
		for some $a_{j,\alpha} \in \NN$. Clearly every monomial in $\mms(d + \alpha)$ can actually be obtained, so necessarily $a_{j,\alpha} \geq 1$. Furthermore, note that there is only one way to get a term from $\mms(d + \alpha)$ when $j = k$, so we can write
		\[\mms(d) \cdot \el{k} = \mms(d_1 + 1, \ldots, d_k + 1, 0, \ldots, 0) + \sum_{j = k+1}^n \sum_{\alpha \in I^d_j} a_{j,\alpha} \mms(d + \alpha).\]
		
		We summarize our conclusions in the following lemma.
		
		\begin{lemma}\label{LEMMA: Factorization of elementary piece (general)}
			Let $X$ be a uc-semiring, $n \in \NN$, $k \in \{1, \ldots, n\}$ and $d = (d_1, \ldots, d_n) \in \NN^n$ a decreasing sequence, for which all non-zero terms appear among the first $k$ terms. Then
			\[\mms(d) \cdot \el{k} = \mms(d_1 + 1, \ldots, d_k + 1, 0, \ldots, 0) + \sum_{j = k+1}^n \sum_{\alpha \in I^d_j} a_{j,\alpha} \mms(d + \alpha)\]
			for some $a_{j,\alpha} \in \NN_{\geq 1}$.
		\end{lemma}
		
		We now restrict ourselves to Frobenius semirings.
		
		\begin{lemma}\label{LEMMA: Variant Frobenius}
			Let $X$ be an upper-bound Frobenius semiring and $n \in \NN_{\geq 2}$. Then
			\[x^n + x^{n-1} y + y^n \ = \ x^n + y^n\]
			for all $x, y \in X$.
		\end{lemma}
		
		\begin{proof}
			Since $X$ is Frobenius, we have
			\[x^n + y^n \leq x^n + x^{n-1} y + y^n \leq \sum_{k = 0}^n \binom{n}{k} x^{n-k} y^k = (x + y)^n = x^n + y^n.\]
			Since $X$ is upper-bound, we conclude $x^n + y^n = x^n + x^{n-1} y + y^n$.
		\end{proof}
		
		In Frobenius semirings we can symplify the formula from Lemma~\ref{LEMMA: Factorization of elementary piece (general)} significantly. To demonstrate how, we first continue the example from before, and afterwards prove the general result.
		
		\begin{example}\label{EXAMPLE: Multiplying a minimal symmetric segment with an elementary symmetric polynomial in a Frobenius semiring}
			Recall from Example~\ref{EXAMPLE: Multiplying a minimal symmetric segment with an elementary symmetric polynomial} that
			\[\mms(3,2,0,0) \cdot \mms(1,1,0,0) = \mms(4,3,0,0) + \mms(4,2,1,0) + 2 \cdot \mms(3,3,1,0) + \mms(3,2,1,1).\]
			We claim that, assuming the Frobenius property, we can eliminate all minimal symmetric segments on the right-hand side save the first one. We start from the back.
			
			The first term in $\mms(3,2,1,1)$ is $x^3 y^2 z w$. We also have terms $x^4 y^2 z$ and $w^4 y^2 z$ in the earlier segment $\mms(4,2,1,0)$. With Lemma~\ref{LEMMA: Variant Frobenius} we can calculate
			\[x^4 y^2 z + x^3 y^2 z w + w^4 y^2 z = y^2 z (x^4 + x^3 w + w^4) = y^2 z (x^4 + w^4) = x^4 y^2 z + w^4 y^2 z.\]
			We have managed to eliminate one term of $\mms(3,2,1,1)$ and with the same kind of reasoning we can eliminate others as well, removing this segment altogether.
			
			Next, consider the term $x^3 y^3 z$ from $\mms(3,3,1,0)$. We can eliminate it with the terms $x^4 y^3$ and $z^4 x^3$ from $\mms(4,3,0,0)$. The segment $\mms(3,3,1,0)$ (and in particular the term $x^3 y^3 z$) is multiplied by $2$, but we can just do the elimination twice. Considering also other relevant permutations of variables, we eliminate the whole $2 \cdot \mms(3,3,1,0)$.
			
			Finally, eliminate $x^4 y^2 z$ in $\mms(4,2,1,0)$ with the help of terms $x^4 y^3$ and $x^4 z^3$ from $\mms(4,3,0,0)$, and similarly for the rest of $\mms(4,2,1,0)$. In the end we get just $\mms(3,2,0,0) \cdot \mms(1,1,0,0) = \mms(4,3,0,0)$.
		\end{example}
		
		We now tackle the general case.
		
		\begin{lemma}\label{LEMMA: Factorization of elementary piece (Frobenius)}
			Let $X$ be an upper-bound Frobenius semiring, $n \in \NN$, $k \in \{1, \ldots, n\}$ and $d = (d_1, \ldots, d_n) \in \NN^n$ a decreasing sequence, for which all non-zero terms appear among the first $k$ terms. Then
			\[\mms(d) \cdot \el{k} = \mms(d_1 + 1, \ldots, d_k + 1, 0, \ldots, 0).\]
		\end{lemma}
		
		\begin{proof}
			By Lemma~\ref{LEMMA: Factorization of elementary piece (general)} we have
			\[\mms(d) \cdot \el{k} = \mms(d_1 + 1, \ldots, d_k + 1, 0, \ldots, 0) + \sum_{j = k+1}^n \sum_{\alpha \in I^d_j} a_{j,\alpha} \mms(d + \alpha)\]
			for some $a_{j,\alpha} \in \NN_{\geq 1}$. If $k = n$, we are done, so assume hereafter that $k < n$.
			
			We claim that we can remove the summands with $j \geq k+1$, first for the largest $j$, then working down. Here is one step of the procedure. Take any $j$ with $k+1 \leq j \leq n$ and any $\alpha \in I^d_j$. Let $i \in \{1, \ldots, n\}$ denote the first index, for which $\alpha_i = 0$. Such $i$ exists because $k < n$, and we have $i < j$ since $j \geq k+1$. Define $\beta$ to be the same finite binary sequence as $\alpha$, except with the $i$-th and $j$-th term switched. Observe that then $\beta \in I^d_{j-1}$.
			
			Consider any monomial $x_{\pi(1)}^{d_1 + \alpha_1} \ldots x_{\pi(n)}^{d_n + \alpha_n}$ in $\mms(d + \alpha)$. Let $\rho$ denote the same permutation as $\pi$, except with the $i$-th and $j$-th value switched. Then both $x_{\pi(1)}^{d_1 + \beta_1} \ldots x_{\pi(n)}^{d_n + \beta_n}$ and $x_{\rho(1)}^{d_1 + \beta_1} \ldots x_{\rho(n)}^{d_n + \beta_n}$ appear in $\mms(d + \beta)$.
			
			Since $a_{j-1,\beta} \geq 1$, we can ``borrow'' these two terms from $a_{j-1,\beta} \, \mms(d + \beta)$; denoting
			\[P \dfeq x_{\pi(1)}^{d_1 + \alpha_1} \ldots x_{\pi(1)}^{d_{i-1} + \alpha_{i-1}} \cdot x_{\pi(i)}^{d_j} \cdot x_{\pi(1)}^{d_{i+1} + \alpha_{i+1}} \ldots x_{\pi(1)}^{d_{j-1} + \alpha_{j-1}} \cdot x_{\pi(j)}^{d_j} \cdot x_{\pi(1)}^{d_{j+1} + \alpha_{j+1}} \ldots x_{\pi(1)}^{d_n + \alpha_n}\]
			and using Lemma~\ref{LEMMA: Variant Frobenius}, we get
			\[x_{\pi(1)}^{d_1 + \beta_1} \ldots x_{\pi(n)}^{d_n + \beta_n} \ \ + \ \ x_{\pi(1)}^{d_1 + \alpha_1} \ldots x_{\pi(n)}^{d_n + \alpha_n} \ \ + \ \ x_{\rho(1)}^{d_1 + \beta_1} \ldots x_{\rho(n)}^{d_n + \beta_n} =\]
			\[= P \cdot \left(x_{\pi(i)}^{d_i - d_j + 1} + x_{\pi(i)}^{d_i - d_j} x_{\pi(j)} + x_{\pi(j)}^{d_i - d_j + 1}\right) = P \cdot \left(x_{\pi(i)}^{d_i - d_j + 1} + x_{\pi(j)}^{d_i - d_j + 1}\right) =\]
			\[= x_{\pi(1)}^{d_1 + \beta_1} \ldots x_{\pi(n)}^{d_n + \beta_n} \ \ + \ \ x_{\rho(1)}^{d_1 + \beta_1} \ldots x_{\rho(n)}^{d_n + \beta_n}.\]
			The term $x_{\pi(1)}^{d_1 + \alpha_1} \ldots x_{\pi(n)}^{d_n + \alpha_n}$ has vanished! If we do that for all terms in $\mms(d + \alpha)$, we have effectively decreased $a_{j,\alpha}$ by one, while only using terms from $(j-1)$-segments (which remained unchanged). Keep doing that for all $j$ between $k+1$ and $n$ from largest to smallest until all those $a_{j,\alpha}$ are reduced to zero. In the end, we are left with just
			\[\mms(d) \cdot \el{k} = \mms(d_1 + 1, \ldots, d_k + 1, 0, \ldots, 0).\]
		\end{proof}
		
		\begin{lemma}\label{LEMMA: Elementary factorization of a segment}
			Let $X$ be an upper-bound Frobenius semiring, $n \in \NN$ and $d = (d_1, \ldots, d_n) \in \NN^n$ a decreasing sequence. Then
			\[\mms(d) = \el{1}^{d_1 - d_2} \cdot \ldots \cdot \el{n-1}^{d_{n-1} - d_n} \cdot \el{n}^{d_n}.\]
		\end{lemma}
		
		\begin{proof}
			Use Lemma~\ref{LEMMA: Factorization of elementary piece (Frobenius)} as an induction step to first factor out $\el{n}$ from $\mms(d)$ as many times as possible (which is $d_n$ many), then $\el{n-1}$ from the remainder as many times as possible (which is $d_{n-1} - d_n$), and so on. In the end we are left with $\mms(0, \ldots, 0)$, which is equal to $1$.
		\end{proof}
		
		We are now ready to prove the central theorem in the paper.
		
		\begin{theorem}[Characterization of elementary upper-bound semirings]\label{THEOREM: Characterization of elementary upper-bound semirings}
			The following statements are equivalent for any upper-bound uc-semiring $X$.
			\begin{enumerate}
				\item
					$X$ is Frobenius.
				\item
					$X$ is fully elementary.
				\item
					$X$ is $n$-elementary for some $n \in \NN_{\geq 2}$.
				\item
					$X$ is $2$-elementary.
			\end{enumerate}
		\end{theorem}
		
		\begin{proof}
			\begin{itemize}
				\item\proven{$(1 \impl 2)$}
					Any symmetric polynomial can be written as a linear combination of minimal symmetric segments, and any such segment can be written as a product of elementary symmetric polynomials by Lemma~\ref{LEMMA: Elementary factorization of a segment}.
				\item\proven{$(2 \impl 3)$}
					By definition.
				\item\proven{$(3 \impl 4)$}
					By~Proposition~\ref{PROPOSITION: Higher elementarity implies lower one}.
				\item\proven{$(4 \impl 1)$}
					Take any $n \in \NN_{\geq 1}$. The polynomial $x^n + y^n$ is symmetric, so by assumption we can write
					\[x^n + y^n = \sum_{k = 1}^m a_k (x y)^{i_k} (x + y)^{j_k}.\]
					This holds for all $x, y \in X$, and we can choose to set $y$ to $0$ and replace $x$ with $x + y$ which gives us
					\[(x + y)^n = \sum_{k \in \{1, \ldots, m\}, i_k = 0} a_k (x + y)^{j_k}.\]
					Hence
					\[x^n + y^n \leq \sum_{k = 0}^n \binom{n}{k} x^{n-k} y^k = (x + y)^n = \sum_{k \in \{1, \ldots, m\}, i_k = 0} a_k (x + y)^{j_k} \leq\]
					\[\leq \sum_{k = 1}^m a_k (x y)^{i_k} (x + y)^{j_k} = x^n + y^n\]
					(note that we needed $n \geq 1$ for the first step in this chain). $X$ is upper-bound, so we conclude $x^n + y^n = (x + y)^n$.
			\end{itemize}
		\end{proof}
		
		In short, an upper-bound uc-semiring is fully elementary if and only if it is Frobenius. It is important that the semiring is upper-bound --- we do not have the equivalence in general. According to the fundamental theorem of symmetric polynomials over rings, all uc-rings are elementary, but they are not Frobenius in general (take for example the ring of real numbers $\RR$). We discuss this more in detail in Section~\ref{SECTION: Towards General Elementarity}.

	\section{Elementarity in Linearly Ordered Semirings}\label{SECTION: Elementarity in Linearly Ordered Semirings}
	
		In the previous section we characterized the fully elementary upper-bound semirings as the Frobenius ones. However, many interesting upper-bound semirings are linearly ordered\footnote{Recall that $X$ is linearly ordered by $\leq$ when all $x, y \in X$ are comparable, \ie $x \leq y$ or $y \leq x$.} by their intrinsic order --- in fact, our starting motivation was to study elementarity in semirings such as the tropical and the extended tropical semiring, which are linearly ordered. In this section we improve the characterization for such semirings.
		
		To set the stage, we first recall the basic facts about the \df{ghost ideal}~\cite{Izhakian20102222, extendedsemiring, KVL2017}, defined for any semiring $X$ as the subset of elements which satisfy the idempotency condition:
		\[\nu{X} \dfeq \st{a \in X}{a = a + a}.\]
		Clearly $0 \in \nu{X}$ and if $a, b \in \nu{X}$, then $a + b \in \nu{X}$. Also, if $a \in \nu{X}$ and $x \in X$, then $x \cdot a \in \nu{X}$ and $a \cdot x \in \nu{X}$. In short, $\nu{X}$ is indeed a semiring ideal in $X$. We can now immediately conclude from the definition that $\nu{X}$ is an idempotent semiring.
		
		Define the map $\nu\colon X \to X$ by $\nu(x) \dfeq x + x$. Clearly $\nu$ is an additive monoid homomorphism, although it is not necessarily a semiring homomorphism (take for example $X = \NN$). Also, $\nu$ is a monotone map (with regard to the intrinsic order) and we have $x \leq \nu(x)$ for all $x \in X$.
		
		By definition, $\nu{X}$ is the set of fixed points of $\nu$. Note that if $X$ is unital (so we have natural numbers in $X$), we can write $\nu(x) = 2x$ and $\nu{X} = \st{a \in X}{a = 2a}$.
		
		The consideration of ghost ideals arose from the study of \df{supertropical semirings}~\cite{Izhakian20102222, extendedsemiring, KVL2017}, which we will get to shortly. But first, a more general notion.
		
		\begin{definition}\label{DEFINITION: Quasiidempotent semirings}
			A semiring $X$ is \df{quasiidempotent} when $x + x + x + x = x + x$ for all $x \in X$.
		\end{definition}
		
		Clearly $\nu$ can be used to characterize quasiidempotent semirings.
		
		\begin{proposition}\label{PROPOSITION: Characterization of quasiidempotent semirings}
			Let $X$ be a semiring. The following statements are equivalent.
			\begin{enumerate}
				\item
					$X$ is quasiidempotent.
				\item
					$\nu$ is a projection (\ie $\nu \circ \nu = \nu$).
				\item
					The image of $\nu$ is $\nu{X}$.
				\item
					The image of $\nu$ is an idempotent subsemiring in $X$.
			\end{enumerate}
		\end{proposition}
		
		\begin{proof}
			Easy.
		\end{proof}
		
		When we have a multiplicative unit, we can extend the characterization of quasiidempotent semirings.
		
		\begin{proposition}\label{PROPOSITION: Characterization of unital quasiidempotent semirings}
			Let $X$ be a unital semiring. The following statements are equivalent.
			\begin{enumerate}
				\item
					$X$ is quasiidempotent.
				\item
					$2 = 4$ in $X$.
				\item
					$\nu$ is a semiring homomorphism.
				\item
					$\nu{X}$ has a multiplicative unit and the restriction $\nu\colon X \to \nu{X}$ is a unital semiring homomorphism.
			\end{enumerate}
		\end{proposition}
		
		\begin{proof}
			\begin{itemize}
				\item\proven{$(1 \impl 2)$}
					Apply quasiidempotency for the unit $1$.
				\item\proven{$(2 \impl 1)$}
					If $2 = 4$, then $2x = 4x$ for any $x \in X$.
				\item\proven{$(2 \impl 3)$}
					We know that $\nu$ is an additive monoid homomorphism. It remains to check that $\nu$ preserves multiplication: $\nu(x) \cdot \nu(y) = 2x \cdot 2y = 4xy = 2xy = \nu(xy)$.
				\item\proven{$(3 \impl 2)$}
					$2 = \nu(1) = \nu(1 \cdot 1) = \nu(1) \cdot \nu(1) = 2 \cdot 2 = 4$.
				\item\proven{$(2 \land 3 \impl 4)$}
					We have $\nu(1) = 2$ and $2 \cdot \nu(x) = 4x = 2x = \nu(x)$.
				\item\proven{$(4 \impl 3)$}
					Yes.
			\end{itemize}
		\end{proof}
		
		Quasiidempotent semirings are useful because they allow us to view $X$ as a bundle\footnote{Here we use the term `bundle' in the most general sense --- as any map, where we treat the domain as a disjoint union of fibers.} over $\nu{X}$, with $\nu$ as the bundle projection. This bundle has a section $\nu{X} \hookrightarrow X$ and both $\nu$ and this inclusion preserve the semiring operations and the intrinsic order. This allows us to pass between the total space and the base of the bundle, often reducing issues from $X$ to the more manageable (since it is an idempotent semiring) $\nu{X}$.
		
		In light of this view, let us examine fibers of $\nu$ closer.
		
		\begin{lemma}\label{LEMMA: Fibers of nu}
			Let $X$ be a quasiidempotent semiring.
			\begin{enumerate}
				\item
					Given $a \in X$, the fiber $\nu^{-1}(a)$ is non-empty if and only if $a \in \nu{X}$. In that case $a$ is the largest element of $\nu^{-1}(a)$.
				\item
					Suppose $X$ is upper-bound. For any $x \in X$ and $a \in \nu{X}$, if $x \leq a$, then $x + a = a$.\footnote{As a special case, when $X$ is idempotent (therefore upper-bound, $1 = 2$ and $\nu{X} = X$), we get that $a \leq b \lequ a + b = b$ for all $a, b \in X$.}
			\end{enumerate}
		\end{lemma}
		
		\begin{proof}
			\begin{enumerate}
				\item
					Suppose we have $x \in \nu^{-1}(a)$; then $2a = 4x = 2x = a$, so $a \in \nu^{-1}(a)$. Conversely, if $a = 2a$, then $a \in \nu^{-1}(a)$.
					
					Take any $x \in \nu^{-1}(a)$. Then $x \leq 2x = a$.
				\item
					Since $x \leq a$, there exists $y \in X$ with $x + y = a$. We then have
					\[a \leq x + a \leq 2x + 2a = 4x + 2y = 2x + 2y = 2a = a.\]
					Since $X$ is upper-bound, we conclude $a = x + a$.
			\end{enumerate}
		\end{proof}
		
		We now recall the definition of supertropical semirings~\cite{Izhakian201661, KVL2017}.
		
		\begin{definition}
			A \df{supertropical semiring} $X$ is a quasiidempotent uc-semiring which satisfies the following for all $a, b \in X$:
			\begin{itemize}
				\item
					if $\nu(a) \neq \nu(b)$, then $a + b \in \{a, b\}$,
				\item
					if $\nu(a) = \nu(b)$, then $a + b = \nu(a) = \nu(b)$.
			\end{itemize}
		\end{definition}
		
		Finally, we get to linearly ordered upper-bound semirings. In fact, the intrinsic order is the only order on semirings that we consider in this paper; thus, any time we refer to a linearly ordered semiring, we mean linearly ordered by its intrinsic order.
		
		In any linearly ordered upper-bound semiring $X$ we can introduce the \df{strict order relation} $<$ as one would expect. For $a, b \in X$ define $a < b$ in either of the following equivalent ways:
		\[a \leq b \land a \neq b \quad \iff \quad \lnot(b \leq a).\]
		The relation $<$ is irreflexive, asymmetric and transitive. Also, $<$ satisfies the law of trichotomy: for any $a, b \in X$ exactly one of the statements $a < b$, $a = b$, $b < a$ holds.
		
		\begin{lemma}\label{LEMMA: Fibers of nu, linear version}
			Let $X$ be a linearly ordered upper-bound quasiidempotent unital semiring.
			\begin{enumerate}
				\item
					For any $a \in \nu{X}$ the fiber $\nu^{-1}(a)$ contains either exactly one or exactly two elements.
				\item
					For all $x, y \in X$ we have $x = y$ if and only if
					\[\nu(x) = \nu(y) \quad\land\quad \big(x \in \nu{X} \iff y \in \nu{X}\big).\]
				\item
					If $x, y \in X$, $y \notin \nu{X}$ and $x < y$, then $2x < y$.
				\item
					For any $x, y \in X$, if $x < y$, then $x + y = y$.
			\end{enumerate}
		\end{lemma}
		
		\begin{proof}
			\begin{enumerate}
				\item
					Since $a \in \nu{X}$, Lemma~\ref{LEMMA: Fibers of nu} tells us that $a$ is the largest element of $\nu^{-1}(a)$. Suppose we had at least two further elements $x, y \in \nu^{-1}(a)$; without loss of generality assume $x < y < a$. Then there exists $u \in X$ such that $x + u = y$. It follows that $2u \leq 2y = a$, but we cannot have $2u = a$ since that would mean $u \in \nu^{-1}(a)$ and then $y = x + u \geq 2 \min\{x,u\} = a$ (since in a linear order $\min\{x,u\} \in \{x,u\}$). Thus $2u < a$. If $x \leq 2u$, then $a = 2x \leq 4u = 2u < a$, a contradiction, so $2u < x$. Then there exists $v \in X$ such that $2u + v = x$. Lemma~\ref{LEMMA: Natural number equalities in a semiring} gives us $2 = 3$, and we get a contradiction $y = x + u = 3u + v = 2u + v = x$.
				\item
					Follows easily from the previous item.
				\item
					Suppose $2x \geq y$; then $2x = 4x \geq 2y$. On the other hand $x \leq y$ implies $2x \leq 2y$, so $2x = 2y$. But $\nu^{-1}(2y) = \{y, 2y\}$ and $x < y < 2y$, a contradiction.
				\item
					If $y \in \nu{X}$, then $x + y = y$ by Lemma~\ref{LEMMA: Fibers of nu}. Assume now $y \notin \nu{X}$; then by the previous item $2x < y$. We have some $u \in X$ with $x + u = y$. Certainly $u \leq y$. Suppose $u < y$; then $2u < y$. Since $2x, 2u \in \nu{X}$, Lemma~\ref{LEMMA: Fibers of nu} tells us that $2y = 2x + 2u \in \{2x, 2u\}$. In either case we get $2y < y$, a contradiction. Hence $u = y$, meaning that $y = x + u = x + y$.
			\end{enumerate}
		\end{proof}
		
		We can now show that when a semiring is linearly ordered, elementarity (as well as supertropicality) reduces to a simple comparison of two numbers.
		
		\begin{theorem}[Characterization of elementary linearly ordered semirings]\label{THEOREM: Characterization of elementary linearly ordered semirings}
			The following statements are equivalent for any linearly ordered upper-bound uc-semiring $X$.
			\begin{enumerate}
				\item
					The equivalent statements of Theorem~\ref{THEOREM: Characterization of elementary upper-bound semirings} hold for $X$ (in particular, $X$ is fully elementary).
				\item
					$X$ is supertropical.
				\item
					$X$ is quasiidempotent.
				\item
					$2 = 3$ in $X$.
			\end{enumerate}
		\end{theorem}
		
		\begin{proof}
			\begin{itemize}
				\item\proven{$(1 \impl 3)$}
					Lemma~\ref{LEMMA: Natural number equalities in a semiring} tells us that the Frobenius property of $X$ implies $2 = 4$ in $X$.
				\item\proven{$(1 \land 3 \impl 2)$}
					Take any $a, b \in X$. If $a = b$, then $a + b = 2a = \nu(a)$. Otherwise without loss of generality $a < b$, so $a + b = b \in \{a, b\}$ by Lemma~\ref{LEMMA: Fibers of nu, linear version}.
				\item\proven{$(2 \impl 3)$}
					By definition.
				\item\proven{$(3 \impl 1)$}
					We prove the Frobenius equality for any $n \in \NN_{\geq 1}$. Take $x, y \in X$; since $X$ is linearly ordered, we may without loss of generality assume $x \leq y$. If $x = y$, then
					\[(x + y)^n = (2x)^n = 2^n x^n = 2 x^n = x^n + y^n.\]
					On the other hand, if $x < y$, then $x + y = y$ by Lemma~\ref{LEMMA: Fibers of nu, linear version}. Hence
					\[x^n + y^n \leq \sum_{k = 0}^n \binom{n}{k} x^{n-k} y^k = (x + y)^n = y^n \leq x^n + y^n.\]
					Since $X$ is upper-bound, we conclude $(x + y)^n = x^n + y^n$.
				\item\proven{$(3 \lequ 4)$}
					By Lemma~\ref{LEMMA: Natural number equalities in a semiring} the statements $2 = 4$ and $2 = 3$ are equivalent in our situation.
			\end{itemize}
		\end{proof}

	\section{Towards General Characterization of Elementarity}\label{SECTION: Towards General Elementarity}
	
		The previous two sections established the main results of the paper. This section, on the other hand, is more speculative. We consider what the scope of our results is --- we look at them from the perspective of general uc-semirings, rather than just the upper-bound ones.
		
		Recall that Theorem~\ref{THEOREM: Characterization of elementary upper-bound semirings} tells us that full elementarity is equivalent to the Frobenius property for upper-bound uc-semirings. Clearly, we do not have such an equivalence for general uc-semirings. By the fundamental theorem of symmetric polynomials over rings, every unital commutative ring is fully elementary. Of course, not every uc-ring is Frobenius.\footnote{There is a reason why the Frobenius equalities $(x + y)^n = x^n + y^n$ are informally called ``Freshman's Dream''.} There is a simple characterization for when it is.
		
		\begin{proposition}
			A uc-ring $X$ is Frobenius if and only if $x \cdot y \cdot (x + y) = 0$ for all $x, y \in X$.
		\end{proposition}
		
		\begin{proof}
			\begin{itemize}
				\item\proven{$(\impl)$}
					Suppose $X$ is a Frobenius ring. We have $2 = 4$ by Lemma~\ref{LEMMA: Natural number equalities in a semiring}, and we can subtract $1$ on both sides to get $1 = 3$. Thus
					\[x^3 + y^3 = (x + y)^3 = x^3 + 3 x^2 y + 3 x y^2 + y^3 = x^3 + x^2 y + x y^2 + y^3 = x^3 + x y (x + y) + y^3.\]
					Subtracting $x^3 + y^3$ on both sides yields the desired equation.
				\item\proven{$(\revimpl)$}
					First apply the assumed equation for $x = y = 1$ to get $2 = 0$. Take now any $n \in \NN_{\geq 1}$; we prove the Frobenius equality for the exponent $n$ by induction on $n$. Of course, the Frobenius equality for $n = 1$ always holds.
					
					As for the induction step, assume that the Frobenius equalities hold for exponents, smaller than $n$ (but larger than $0$). We have $(x + y)^n = \sum_{k = 0}^n \binom{n}{k} x^{n-k} y^k$. If $n$ is even, we have the middle binomial coefficient $\binom{n}{n/2} = \binom{n-1}{n/2 - 1} + \binom{n-1}{n/2} = 2 \cdot \binom{n-1}{n/2}$ which is even, so equal to $0$ in $X$. All other terms (regardless of the parity of $n$) come in pairs, and for $0 < k < n/2$ we have
					\[\binom{n}{k} x^{n-k} y^k + \binom{n}{n-k} x^k y^{n-k} = \binom{n}{k} (x y)^k (x^{n-2k} + y^{n-2k}) =\]
					\[= \binom{n}{k} (x y)^k (x + y)^{n-2k} = x y (x + y) \binom{n}{k} (x y)^{k-1} (x + y)^{n-2k-1} = 0.\]
					Thus what we are left with is $(x + y)^n = \binom{n}{0} x^n + \binom{n}{n} y^n = x^n + y^n$.
			\end{itemize}
		\end{proof}
		
		Thus not every fully elementary semiring is Frobenius. What about the other implication? We do not yet know the answer, so we pose this as a question.
		
		\begin{question}
			Is every Frobenius semiring fully elementary?
		\end{question}
		
		The problem with answering such a question is that, at least in rings, and presumably in many other semirings as well, the Frobenius property is entirely incidental to elementarity. On the other hand, it proved to be crucial for upper-bound semirings. We now make it explicit why this was the case.
		
		\begin{lemma}\label{LEMMA: Symhomomorphic semirings}
			Let $X$ be a uc-semiring. The following statements are equivalent.
			\begin{enumerate}
				\item
					$X$ is Frobenius and $2 = 3$ in $X$.
				\item
					$X$ satisfies the statement of Lemma~\ref{LEMMA: Variant Frobenius}.
				\item
					$X$ satisfies the statement of Lemma~\ref{LEMMA: Factorization of elementary piece (Frobenius)}
				\item
					$X$ satisfies the statement of Lemma~\ref{LEMMA: Elementary factorization of a segment}
				\item
					For every $n \in \NN$ the map $\mms$ is a monoid homomorphism from the additive monoid of decreasing sequences of length $n$ to the multiplicative monoid of polynomial functions in $n$ variables over $X$.
			\end{enumerate}
		\end{lemma}
		
		\begin{proof}
			\begin{itemize}
				\item\proven{$(1 \impl 2)$}
					We reprove Lemma~\ref{LEMMA: Variant Frobenius}, but this time only under the assumption that $X$ is Frobenius and $2 = 3$ in $X$.
					
					Take any $x, y \in X$ and $n \in \NN_{\geq 2}$. Then $\binom{n}{1} = n = n + 1$ in $X$. By Frobenius we get
					\[x^n + y^n = (x + y)^n = \sum_{k = 0}^n \binom{n}{k} x^{n-k} y^k =\]
					\[= x^{n-1} y + \sum_{k = 0}^n \binom{n}{k} x^{n-k} y^k = x^{n-1} y + (x + y)^n = x^n + x^{n-1} y + y^n.\]
				\item\proven{$(2 \impl 3)$}
					Note that the proof of Lemma~\ref{LEMMA: Factorization of elementary piece (Frobenius)} does not use the fact that $X$ is upper-bound directly. It relies on Lemma~\ref{LEMMA: Factorization of elementary piece (general)} (applicable for general uc-semirings) and Lemma~\ref{LEMMA: Variant Frobenius}.
				\item\proven{$(3 \impl 4)$}
					The proof of Lemma~\ref{LEMMA: Elementary factorization of a segment} does not use the upper-boundedness of $X$ directly. It relies on Lemma~\ref{LEMMA: Factorization of elementary piece (Frobenius)}.
				\item\proven{$(4 \impl 5)$}
					By assumption we have $\mms(d) = \el{1}^{d_1 - d_2} \cdot \ldots \cdot \el{n-1}^{d_{n-1} - d_n} \cdot \el{n}^{d_n}$ for any decreasing $d \in \NN^n$. Applying this for $d'$, $d''$, $d' + d''$, we get
					\[\mms(d' + d'') =  \el{1}^{(d'_1 + d''_1) - (d'_2 + d''_2)} \cdot \ldots \cdot \el{n-1}^{(d'_{n-1} + d''_{n-1}) - (d'_n + d''_n)} \cdot \el{n}^{d'_n + d''_n} =\]
					\[= \big(\el{1}^{d'_1 - d'_2} \cdot \ldots \cdot \el{n-1}^{d'_{n-1} - d'_n} \cdot \el{n}^{d'_n}\big) \cdot \big(\el{1}^{d''_1 - d''_2} \cdot \ldots \cdot \el{n-1}^{d''_{n-1} - d''_n} \cdot \el{n}^{d''_n}\big) = \mms(d') \cdot \mms(d'').\]
				\item\proven{$(5 \impl 1)$}
					Take any $n \in \NN_{\geq 1}$ and $x, y \in X$. We get\footnote{One can make the case that this argument is the essence, why we do not include $n = 0$ in the Frobenius property (Definition~\ref{DEFINITION: Frobenius}): $\mms(n,0)$ has two terms for $n \geq 1$, and only one for $n = 0$.}
					\[x^n + y^n = \mms(n,0) = \mms\big((1,0) + \ldots + (1,0)\big) = \mms(1,0) \cdot \ldots \cdot \mms(1,0) = \big(\mms(1,0)\big)^n = (x + y)^n.\]
					Thus $X$ is Frobenius, and Lemma~\ref{LEMMA: Natural number equalities in a semiring} tells us that therefore natural numbers, greater than or equal to $2$, are the same in $X$ if they have equal parity. Hence
					\[2 = 6 = \mms(2,1,0)(1,1,1) = \mms(1,0,0)(1,1,1) \cdot \mms(1,1,0)(1,1,1) = 3 \cdot 3 = 9 = 3.\]
			\end{itemize}
		\end{proof}
		
		In view of the last of the equivalent statements of the lemma, we introduce the following definition.
		
		\begin{definition}\label{DEFINITION: Symhomomorphic semirings}
			A uc-semiring is \df{symhomomorphic} when it satisfies the equivalent statements of Lemma~\ref{LEMMA: Symhomomorphic semirings}.
		\end{definition}
		
		The point is that it is the symhomomorphic property which captures the essence of the idea of the proof of elementarity in this paper. In fact, we can distill the results of Section~\ref{SECTION: Elementarity in Upper-bound Semirings} into the following corollary.
		
		\begin{corollary}
			Let $X$ be a uc-semiring.
			\begin{enumerate}
				\item
					If $X$ is upper-bound and Frobenius, it is symhomomorphic.
				\item
					If $X$ is symhomomorphic, it is fully elementary.
			\end{enumerate}
		\end{corollary}
		
		\begin{proof}
			\begin{enumerate}
				\item
					By Lemmas~\ref{LEMMA: Natural number equalities in a semiring}~and~\ref{LEMMA: Symhomomorphic semirings}.
				\item
					The proof of $(1 \impl 2)$ in Theorem~\ref{THEOREM: Characterization of elementary upper-bound semirings} relies only on Lemma~\ref{LEMMA: Elementary factorization of a segment}.
			\end{enumerate}
		\end{proof}
		
		Of course, this is just the part that Frobenius implies full elementarity. Theorem~\ref{THEOREM: Characterization of elementary upper-bound semirings} states that for upper-bound semirings the converse is also true. How critical is the upper-boundedness for this argument?
		
		\begin{question}
			For upper-bound uc-semirings the Frobenius property and full elementarity are equivalent. Is there a larger natural class of semirings where this is still true?
		\end{question}
		
		At first glance one might be optimistic and speculate that between characterizations of full elementarity in upper-bound uc-semirings and in uc-rings, one will get a characterization for general uc-semirings. The point is that upper-bound semirings are in a sense ``orthogonal'' to rings. We have already mentioned that the only upper-bound ring is the trivial one. Moreover, for every semiring $X$ (and the aforementioned equivalence relation $a \approx b \iff a \leq b \land b \leq a$) we have a short exact sequence
		\[\{0\} \to [0] \hookrightarrow X \twoheadrightarrow X/_\approx \to \{0\}\]
		where $X/_\approx$ is an upper-bound semiring and the equivalence class $[0] = \st{x \in X}{0 \approx x}$ is a ring. That is, any uc-semiring can be split into an upper-bound semiring and a uc-ring in this way.
		
		Unfortunately, this approach turns out to be rather limited. Exact sequences are useful for categories which are abelian (or at least additive), which semirings decidedly are not. In particular, it is not even true that if $[0]$ is trivial, $X$ is necessarily upper-bound.
		
		A simple example is $\NN_{2=4} \dfeq \NN/_\equ$ where $\NN$ has the usual semiring operations and $\equ$ is the smallest congruence such that $2 \equ 4$ (that is, we identify any numbers $\geq 2$ which have the same parity). The underlying set of $\NN_{2=4}$ is $\{[0], [1], [2], [3]\}$, and we have $[0] \leq [1] \leq [2] \leq [3] \leq [2]$. The uc-semiring $\NN_{2=4}$ is Frobenius, but not upper-bound or symhomomorphic.
		
		Is it fully elementary? It turns out that the answer is yes. Recall from Lemma~\ref{LEMMA: Factorization of elementary piece (general)} that
		\[\mms(d) \cdot \el{k} = \mms(d_1 + 1, \ldots, d_k + 1, 0, \ldots, 0) + \sum_{j = k+1}^n \sum_{\alpha \in I^d_j} a_{j,\alpha} \mms(d + \alpha).\]
		In $\NN_{2=4}$ we can transfer the big sum to the other side of the equation:
		\[\mms(d_1 + 1, \ldots, d_k + 1, 0, \ldots, 0) = \mms(d) \cdot \el{k} + \sum_{j = k+1}^n \sum_{\alpha \in I^d_j} a_{j,\alpha} \mms(d + \alpha)\]
		(it turns out that when a term in the big sum is not zero, the numbers are large enough that only parity matters).
		
		From here a suitable induction allows us to reduce all minimal symmetric segments to polynomials in elementary symmetric polynomials. In fact, the argument is essentially the same as in the fundamental theorem of symmetric polynomials over rings, where we get
		\[\mms(d_1 + 1, \ldots, d_k + 1, 0, \ldots, 0) = \mms(d) \cdot \el{k} - \sum_{j = k+1}^n \sum_{\alpha \in I^d_j} a_{j,\alpha} \mms(d + \alpha)\]
		and proceed from there.
		
		Recall from Lemma~\ref{LEMMA: Factorization of elementary piece (Frobenius)} that in the case of Frobenius upper-bound uc-semirings we were able to write
		\[\mms(d_1 + 1, \ldots, d_k + 1, 0, \ldots, 0) = \mms(d) \cdot \el{k} + 0 \cdot \sum_{j = k+1}^n \sum_{\alpha \in I^d_j} a_{j,\alpha} \mms(d + \alpha).\]
		Is the ability to write these expressions in these forms merely a coincidence or is there something more to it?
		
		\begin{question}
			Is it the case for every fully elementary semiring $X$ that we have
			\[\mms(d_1 + 1, \ldots, d_k + 1, 0, \ldots, 0) = \mms(d) \cdot \el{k} + u \cdot \sum_{j = k+1}^n \sum_{\alpha \in I^d_j} a_{j,\alpha} \mms(d + \alpha)\]
			for some $u \in X$?
		\end{question}
		
		Above we made the case for this question by examples, but a theoretical case can also be made. If we drop the upper-boundedness condition, the proof of $(4 \impl 1)$ in Theorem~\ref{THEOREM: Characterization of elementary upper-bound semirings} still tells us that $x^n + y^n \approx (x + y)^n$ for all $n \in \NN$ and $x, y \in X$ where $X$ is a fully elementary (or at least $2$-elementary) uc-semiring. Take $n = 2$ and $x = y = 1$ to get $2 \approx 4$. Hence $2 \leq 3 \leq 4 \leq 2$, so also $2 \approx 3$. Thus there exists $u \in X$ such that $3 + u = 2$.
		
		The idea is that this kind of $u$ might ``play the role of $-1$'' sufficiently well. The solution to the equation $3 + u = 2$ may not be unique, but with some luck, given any
		\[\mms(d) \cdot \el{k} = \mms(d_1 + 1, \ldots, d_k + 1, 0, \ldots, 0) + \sum_{j = k+1}^n \sum_{\alpha \in I^d_j} a_{j,\alpha} \mms(d + \alpha)\]
		(as per Lemma~\ref{LEMMA: Factorization of elementary piece (general)}), we could find a solution which allows us to take the big sum to the other side of the equation:ž\[\mms(d_1 + 1, \ldots, d_k + 1, 0, \ldots, 0) = \mms(d) \cdot \el{k} + u \cdot \sum_{j = k+1}^n \sum_{\alpha \in I^d_j} a_{j,\alpha} \mms(d + \alpha).\]
		From here, we could derive full elementarity by induction.
		
		Recall that Theorem~\ref{THEOREM: Characterization of elementary linearly ordered semirings} says (among other things) that a linearly ordered upper-bound uc-semiring is fully elementary if and only if $2 = 3$ holds in it. A positive answer to the following question would be a significant generalization.
		
		\begin{question}
			Is the condition $2 \approx 3$ in a uc-semiring not just a necessary, but also a sufficient condition for full elementarity?
		\end{question}
		
		Of course, this entire discussion is just an aspect of a more general question.
		
		\begin{question}
			Is there a meaningful theorem which classifies elementarity for arbitrary uc-semirings?
		\end{question}

		\printbibliography

\end{document}